\documentclass{article}  
\usepackage[pdftex,
bookmarksnumbered,
bookmarksopen,
colorlinks,
citecolor=blue,
linkcolor=blue,]{hyperref}
\usepackage{amsmath}
\usepackage{amssymb}
\usepackage{mathrsfs}
\usepackage{amsfonts}
\usepackage{amsthm}
\usepackage{algorithm}
\usepackage{booktabs}
\usepackage{listings}
\usepackage{boxedminipage}
\usepackage{algorithmic}
\usepackage{stmaryrd}
\usepackage{cite}
\usepackage{relsize}
\usepackage{lscape}
\usepackage[top=1.3in, bottom=1.5in, left=1.5in, right=1.5in]{geometry}
\usepackage{stmaryrd} 
\usepackage{relsize}%
\usepackage{url}
\usepackage{color,xcolor}

\newtheorem{theorem}{Theorem}[section]
\newtheorem{definition}{Definition}[section]
\newtheorem{proposition}{Proposition}[section]

\newtheorem{lemma}{Lemma}[section]

\newtheorem{remark}{Remark}[section]
\newtheorem{example}{Example}[section]

 \newcommand{\tprod}{\ensuremath{  *     }}   
\newcommand{\unfold}[1]{\ensuremath{ {\rm unfold}\left(  #1 \right)  }} 
\newcommand{\fold}[1]{\ensuremath{ {\rm fold}\left(  #1 \right)  }} 
\newcommand{\bcirc}[1]{\ensuremath{ {\rm bcirc}\left(  #1 \right)  }} 
\newcommand{\vcirc}[1]{\ensuremath{ {\rm circ}\left(  #1 \right)  }} 
\newcommand{\spec}[1]{\ensuremath{ {\rm spec}\left(  #1 \right)  }} 
\newcommand{\Tnnp}{\ensuremath{ \mathbb R^{n\times n \times p}   }} 
\newcommand{\Kp}{\ensuremath{  \mathbb K_p    }} 
\newcommand{\Kpp}{\ensuremath{  \mathbb K_{p+}    }} 
\newcommand{\Kppp}{\ensuremath{  \mathbb K_{p++}    }} 
\newcommand{\Kpppp}{\ensuremath{  \mathbb K_{p+++}    }} 
\newcommand{\Knp}{\ensuremath{ \mathbb K^n_p  }} 
\newcommand{\Knpp}{\ensuremath{ \mathbb K^n_{p+}  }} 
\newcommand{\Knppp}{\ensuremath{ \mathbb K^n_{p++}  }} 
\newcommand{\Knpppp}{\ensuremath{ \mathbb K^n_{p+++}  }} 
\newcommand{\Knnpppp}{\ensuremath{ \mathbb K^{n\times n}_{p+++}  }} 
\newcommand{\Knnppp}{\ensuremath{ \mathbb K^{n\times n}_{p++}  }} 
\newcommand{\Knnpp}{\ensuremath{ \mathbb K^{n\times n}_{p+}  }} 
\newcommand{\Knnp}{\ensuremath{ \mathbb K^{n\times n}_{p}  }} 

\newcommand{\Identity}{\ensuremath{ \mathcal I_{nnp}  }} 

    \newcommand{\bigxiaokuohao}[1]{\ensuremath{ \left(  #1 \right) }}      
    \newcommand{\bigjueduizhi}[1]{\ensuremath{ \left|  #1 \right| }}   
    \newcommand{\bigdakuohao}[1]{\ensuremath{ \left\{  #1 \right\} }}         
    \newcommand{\bigzhongkuohao}[1]{\ensuremath{ \left[   #1 \right] }}

	\definecolor{darkgray}{rgb}{0.66, 0.66, 0.66}

\title{A Study on Nonnegative Tubal Matrices
}

\author{  Yuning Yang\thanks{College of Mathematics and Information Science, Guangxi University, Nanning, 530004, China.}
	\thanks{Corresponding author: Yuning Yang. Email: yyang@gxu.edu.cn}        \and Junwei Zhang$^*$                        
}

\begin{document} 
\maketitle

\begin{abstract}
	
Tubal scalars are   usual vectors, and tubal matrices are matrices with every element being a tubal scalar. Such a matrix is often recognized as a third-order tensor. The product between tubal scalars, tubal vectors, and tubal matrices can be done by the powerful t-product. 
In this paper, 	we define nonnegative/positive/strongly positive tubal scalars/vectors/matrices, and establish several properties that are analogous to their matrix counterparts. 
 In particular, we introduce the irreducible tubal matrix, and provide two equivalent characterizations. Then, the celebrated Perron-Frobenius theorem is established on the nonnegative irreducible tubal matrices.  
We show that some conclusions of the PF theorem for nonnegative irreducible matrices can be generalized to the tubal matrix setting, while some are not. One reason is the defined positivity here has a different meaning to its usual sense.  For those conclusions that can not be extended, weaker conclusions are proved. We also show that, if the nonnegative irreducible tubal matrix contains a strongly positive tubal scalar, then most conclusions of the matrix PF theorem hold.

\end{abstract}

\hspace{2mm}\vspace{3mm}

\section{Introduction}

The celebrated Perron-Frobenius theorem is a fundamental result in the nonnegative matrix theory. 
   Due to the study of higher-order tensors, the PF theorem was extended to the largest H-eigenvalue of a  nonnegative irreducible tensor first by Lim \cite{Lim2005}, and later the result was strengthened  by Chang, Pearson, and Zhang \cite{cpz2008}, which is quite close to its matrix counterpart.  Friedland, Gaubert, and Han extended the PF theorem to nonnegative multilinear forms \cite{fgh2010}. Chang, Pearson, and Zhang generalized the PF theorem to the Z-eigenvalue setting \cite{cpz2013}. 
    Besides, 
 a variety of research was focused on studying properties of nonnegative tensors, see e.g., \cite{cpz2010,hhq2011,yy2010}.

 A tubal scalar is a usual vector, a tubal vector is a vector whose every element is a tubal scalar, and a tubal matrix is a matrix whose every element is also a tubal scalar. Such a matrix is often recognized as a third-order tensor. In \cite{kilmer2011factorization,braman2010third,kilmer2013third}, Kilmer, Martin, Braman, and their coauthors introduced the t-prodct between third-order tensors (tubal matrices), enabling one to study the  properties of third-order tensors, such as transpose, inverse, orthogonality, range and kernel spaces, and SVD, as their matrix counterpart. Further properties are studied then. For example, Miao, Qi, and Wei introduced the t-eigenvalue, t-Jordan form, and t-Drazin inverse \cite{miao2021t}.  The tensor t-functions were studied in \cite{miao2020generalized,lund2020tensor}.  The t-positive (semi)definiteness tensors and related properties were investigated by Zheng, Huang, and Wang \cite{zheng2021t}. The    t-eigenvalue and t-eigenvector were studied by Liu and Jin \cite{liu2020study}. The SVD was deeply investigated for tubal matrices by Qi and Luo \cite{qi2021tubal}, where the definition of tubal matrices there is   more general than the current setting.
 
 In this paper, we still use the notion of a tubal matrix to represent a third-order tensor for convenience.
 We first define nonnegative/positive/strongly positive tubal scalars/vectors/matrices. The meaning of some of these notions is different from the usual sense. For instance, a tubal matrix with one frontal slice being a positive matrix while the other frontal slices being zero matrices is called positive in the current definition, while such a tubal matrix is only recognized as a nonnegative third-order tensor before. Using the t-product, we derive several properties of nonnegative tubal matrices. In particular, the concept of reducibility/irreducibility is extended to the tubal matrix case. In our definition, even a tubal matrix is irreducible, its block curculant representation can be reducible. Such a definition is also   different from that in \cite{cpz2008} for tensors. As will be shown in Proposition \ref{prop:two_def_irreducibility} and Remark \ref{rmk:irreducible_wider}, in the third-order tensor case, our definition of irreducibility covers a wider range of tensors than that of \cite{cpz2008}. 
  Similar to the matrix case, equivalent characterizations of the reducibility/irreducibility based on permutations or tensor powers are provided. We then study the PF theorem for nonnegative irreducible tubal matrices. Some conclusions of the matrix PF theorem can be generalized, while some are not. For those that can not be extended, we prove weaker results. We also prove that, if the nonnegative irreducible tubal matrix contains a strongly positive tubal scalar, then most conclusions of the matrix PF theorem hold. 

The rest is organized as follows. Sect. \ref{sec:pre} introduce preliminaries concerning the t-product. Nonnegative tubal scalars/vectors/matrices and irreducible tubal matrices are defined in Sect. \ref{sec:nonnegative}, with properties given. Sect. \ref{sec:pf_thm} studies the PF theorem for nonnegative irreducible tubal matrices.

\section{Preliminaries}\label{sec:pre}

The notations and definitions throughout this paper mainly follow those of \cite{kilmer2011factorization,braman2010third,kilmer2013third}. 

 The tensors considered are in the field $\Tnnp$. We call such a type of tensors square, namely, its first and second modes have equal dimension. 
 The $i$-th frontal slice of $\mathcal A\in\Tnnp$ is denoted as $\mathcal A^{(i)} 
 \in\mathbb R^{n\times n\times 1}$, the $j$-th  horizontal slice is denoted as $\mathcal A_{j,} \in\mathbb R^{1\times n\times p}$, and the $k$-th lateral slice is denoted as $ \mathcal  A_{,k}\in\mathbb R^{n\times 1\times p}$. $\mathbf a_{j,k} \in\mathbb R^{1\times 1\times p}$ represents the $j,k$-th tube of $\mathcal A$ and $\mathbf a^{(i)}_{j,k}$ is the $i$-th entry (frontal slice) of $\mathbf a_{j,k}$.
  In Matlab, the above notations are respectively given as $\mathcal A^{(i)} = \mathcal A(:,:,i), \mathcal A_{j,} = \mathcal A(j,:,:). \mathcal A_{,k} = \mathcal A(:,k,:)$, and $\mathbf a_{j,k} = \mathcal A(j,k,:)$;
 
An element $\mathbf a\in\mathbb R^{1\times 1\times p}$ is recognized as a tubal scalar of length $p$. The space of all such tubal scalars is denoted as $\Kp$. Similarly, an element $A = \bigxiaokuohao{\mathbf a_j}\in\mathbb R^{n\times 1\times p}$ is called a tubal vector consisting of $n$ tubal scalars. The space of all such tubal vectors is denoted as $\Knp$. Likewise, an element $\mathcal A = \bigxiaokuohao{\mathbf a_{j,k}}\in \Tnnp$ can be seen as a tubal matrix consisting of $n$ tubal vectors. We denote by $\Knnp=\Tnnp$. For tubal scalars, vectors, and matrices, the superscript $\cdot^{(i)}$ means the $i$-th frontal slice of this scalar, vector, or matrix, respectively. 

Let $\mathbf a \in\Kp$. Define
\begin{align}
	\label{def:circ_unfold}
	\vcirc{\mathbf a}:= \left[\begin{matrix}
	\mathbf	a^{(1)} &\mathbf a^{(p)} &\mathbf a^{(p-1)} &\cdots &\mathbf a^{(2)} \\
	\mathbf	a^{(2)} &\mathbf a^{(1)} &\mathbf a^{(p)} & \cdots &\mathbf a^{(3)} \\
		\vdots & \ddots & \ddots & \ddots &\vdots \\
	\mathbf	a^{(p)} &\mathbf a^{(p-1)} & \cdots & \mathbf a^{(2)} & \mathbf a^{(1)}
	\end{matrix}\right].
\end{align}
Given $\mathbf a,\mathbf b\in \Kp$,   the multiplication $*$  between $\mathbf a$ and $\mathbf b$ is defined as \cite{braman2010third}
\[
\mathbf a*\mathbf b := \vcirc{\mathbf a}\mathbf b = \left[\begin{matrix}
	\mathbf	a^{(1)} &\mathbf a^{(p)} &\mathbf a^{(p-1)} &\cdots &\mathbf a^{(2)} \\
	\mathbf	a^{(2)} &\mathbf a^{(1)} &\mathbf a^{(p)} & \cdots &\mathbf a^{(3)} \\
	\vdots & \ddots & \ddots & \ddots &\vdots \\
	\mathbf	a^{(p)} &\mathbf a^{(p-1)} & \cdots & \mathbf a^{(2)} & \mathbf a^{(1)}
\end{matrix}\right]  \left[\begin{matrix}
\mathbf	b^{(1)}  \\
\mathbf	b^{(2)}  \\
\vdots   \\
\mathbf	b^{(p)}  
\end{matrix}\right] .
\]
It follows from \cite{braman2010third} that $(\Kp, +,*)$ is a commutative ring with unity, where $+$ is the usual addition. The unity is $\mathbf e$, the tubal scalar whose first entry is $1$ while the other ones are zero. For any $\mathbf a\in \Kp$, $\mathbf e\tprod\mathbf a = \mathbf a$.   Let $\boldsymbol{0}$ denote the zero tubal scalar, i.e., every entry of $\boldsymbol{0}$ is zero. Then for any $\mathbf a\in \Kp$, $\boldsymbol{0}\tprod\mathbf a = \boldsymbol{0}$.

Let $\mathcal A\in \Knnp$. Define the following notations:
\begin{align}
	\label{def:bcirc_unfold}
	\bcirc{\mathcal A}:= \left[\begin{matrix}
	\mathcal 	A^{(1)} & \mathcal A^{(p)} &\mathcal A^{(p-1)} &\cdots & \mathcal A^{(2)} \\
	\mathcal	A^{(2)} &\mathcal A^{(1)} & \mathcal A^{(p)} & \cdots &\mathcal A^{(3)} \\
		\vdots & \ddots & \ddots & \ddots &\vdots \\
	\mathcal	A^{(p)} & \mathcal A^{(p-1)} & \cdots &\mathcal A^{(2)} & \mathcal A^{(1)},
	\end{matrix}\right], 
	\unfold{\mathcal A} = \left[ \begin{matrix}
		A^{(1)}\\
		A^{(2)}\\
		\vdots\\
		A^{(p)}
	\end{matrix}\right].
\end{align}
The fold operator satisfies
\[
\fold{\unfold{\mathcal A}} = \mathcal A. 
\]
Given $\mathcal A,\mathcal B\in \Knnp$,   the t-product $*$ between $\mathcal A$ and $\mathcal B$ is defined as \cite{kilmer2011factorization}
\[
\mathcal A*\mathcal B := \fold{ \bcirc{\mathcal A}\unfold{\mathcal B}  }. 
\]
The t-product can be implemented fast by using the fast Fourier transform \cite{kilmer2011factorization}. 

Similar to the usual product between matrices, the $j$-th horizontal slice of $\mathcal A*\mathcal B$ is equal to the t-product between the $j$-th horizontal slice of $\mathcal A$ and $\mathcal B$:
\begin{align*}\label{eq:t_prod_horizontal_tensor}
\bigxiaokuohao{ \mathcal A*\mathcal B  }_{j,} = \mathcal A_{j,}*\mathcal B.
\end{align*}
Correspondingly, for $A\in\mathbb R^{1\times n\times p} $ and $B\in\mathbb R^{n\times 1\times p}$, which are respectively represented as $A=[\mathbf a_1,\ldots,\mathbf a_n ]$ and
\[
B =  \left[\begin{matrix}
	\mathbf b_1\\
	\vdots\\
	\mathbf b_n
\end{matrix} \right],
\] where $\mathbf a_i,\mathbf b_i\in \Kp$,  there holds
\begin{align*}
	\label{eq:t_prod_tubal_vectors}
	A*B = \sum^n_{i=1}\nolimits \mathbf a_i*\mathbf b_i. 
\end{align*}

For $\mathcal A_1,\mathcal A_2\in \Knnp$, if they are partitioned as
\[
\mathcal A_1 = \bigzhongkuohao{\begin{matrix}
		\mathcal B & \mathcal C\\
		\mathcal D & \mathcal E\\
\end{matrix}},~\mathcal A_2 = \bigzhongkuohao{\begin{matrix}
\mathcal F & \mathcal G\\
\mathcal H & \mathcal I
\end{matrix}},
\]
where $\mathcal B,\mathcal F\in\mathbb K_p^{n_1\times n_1}$, $\mathcal E,\mathcal I\in \mathbb K_p^{n_2\times n_2}$, $\mathcal C,\mathcal G \in \mathbb K_p^{n_1\times n_2}$, and $\mathcal D,\mathcal H\in \mathbb K_p^{n_2\times n_1}$, with $n_1+n_2=n$, then \cite{miao2020generalized}
\[
\mathcal A_1\tprod\mathcal A_2 = \bigzhongkuohao{\begin{matrix}
		\mathcal B\tprod\mathcal F + \mathcal C\tprod\mathcal H & \mathcal B\tprod\mathcal G + \mathcal C\tprod\mathcal I\\
		\mathcal D\tprod\mathcal F + \mathcal E\tprod\mathcal H & \mathcal D\tprod\mathcal G + \mathcal E\tprod\mathcal I
\end{matrix}}. 
\]
In particular, for $\mathcal A = (\mathbf a_{i,j})\in \Knnp$ and $X=(\mathbf x_i)\in \Knp$,
\[
\bigxiaokuohao{\mathcal A\tprod X}_i = \sum^n_{j=1}\nolimits\mathbf a_{i,j}\tprod\mathbf x_j.
\]
 
The transpose of $\mathbf a\in \Kp$ ($A\in \Knp$, or  $\mathcal A\in\Knnp$) is given by transposing each of the frontal slices and then reversing the order of transposed frontal slicers $2$ through $p$, namely,
\begin{align*}
	\mathbf a^{\top}  :=  \fold{ \left[\begin{matrix}
			\mathbf a^{(1)}\\
			\mathbf a^{(p)} \\
			\vdots\\
			\mathbf  a^{(2)}
		\end{matrix} \right]  },
	A^{\top}  := \fold{ \left[\begin{matrix}
		(A^{(1)})^{\top}\\
		(A^{(p)})^{\top} \\
		\vdots\\
		(A^{(2)})^{\top}
	\end{matrix} \right]  },	
	\mathcal A^{\top}  := \fold{ \left[\begin{matrix}
		(\mathcal A^{(1)})^{\top}\\
		(\mathcal A^{(p)})^{\top} \\
		\vdots\\
		(\mathcal A^{(2)})^{\top}
	\end{matrix} \right]  }.
\end{align*}
They are respectively of sizer $1\times 1\times p$, $1\times n\times p$, and $n\times n\times p$. 
There holds the relation:
\begin{align}
	\label{prop:transpose}
	\bcirc{\mathcal A^{\top}} = \bcirc{\mathcal A}^{\top}.
\end{align}
If $\mathcal A,\mathcal B\in \Knnp$, then $\bigxiaokuohao{\mathcal A\tprod\mathcal B}^{\top} = \mathcal B^{\top}\tprod\mathcal A^{\top}$.

The identity tensor (tubal matrix)  in $\Knnp$ is the tensor whose first frontal slice is the identity matrix, and whose other frontal slices are all zeros. For any $\mathcal A\in \Knnp$, $\Identity\tprod\mathcal A=\mathcal A\tprod\Identity = \mathcal A$.

The permutation tensor (tubal matrix) in our context is slightly different from \cite{kilmer2011factorization}. We call $\mathcal P\in \Knnp$ a permutation tensor, if its first frontal slice is a permutation matrix, and other frontal matrices are all zeros. For any $\mathcal A\in \Knnp$, it holds that
\begin{align}
	\label{def:permutation_tensor}
	\mathcal P\tprod	\mathcal A  \tprod\mathcal P^{\top}:= \fold{ \left[\begin{matrix}
		\mathcal P^{(1)}	\mathcal A^{(1)} \bigxiaokuohao{\mathcal P^{(1)}}^{\top}   \\
		\mathcal P^{(1)}		\mathcal A^{(2)}  \bigxiaokuohao{\mathcal P^{(1)}}^{\top}  \\
			\vdots\\
		\mathcal P^{(1)}		\mathcal A^{(p)}  \bigxiaokuohao{\mathcal P^{(1)}}^{\top} 
		\end{matrix} \right]  }.
\end{align}
\begin{proposition}
	\label{prop:permutation_tensor}
	Let $\mathcal P\in \Knnp$ be a permutation tubal matrix. Then $\mathcal P^{\top}\tprod\mathcal P = \mathcal P\tprod\mathcal P^{\top}=\Identity$.
\end{proposition}

The following proposition shows that the permutation tubal matrix acting on a tubal vector performs similar to that in the matrix case.
\begin{proposition}
	\label{prop:permutation_tensor_on_tubal_vector}
	Let $X=(\mathbf x_i)\in\Knp$ with some  $\mathbf x_i=\boldsymbol{0}$. Then there exists a permutation tubal matrix $\mathcal P$, such that $\mathcal P\tprod X$ takes the form $\bigzhongkuohao{ \begin{matrix}
			X_1 \\
			\boldsymbol{0}
	\end{matrix}  },$
where $\boldsymbol{0}$ in this tubal vector denotes a zero tubal vector. 
\end{proposition}
\begin{proof}
	The assumption shows that every frontal slice $X^{(i)} \in \mathbb R^n$ of $X$ contains zero entries, some of which appear at the same positions. Then there exists a permutation matrix $P\in\mathbb R^{n\times n}$, such that 
	$PX^{(i)} = \bigzhongkuohao{\begin{matrix}
			X^{(i)}_1\\
			\boldsymbol{0}
	\end{matrix}} \in \mathbb R^n$, $i=1,\ldots,p$, where $\boldsymbol{0}$ here means a zero vector. Now, define $\mathcal P$ such that its first frontal   slice is $P$, and other frontal slices are zero matrices. The required result follows immediately.
\end{proof}

%
%
%
%

\section{Nonnegative Tubal Scalars, Vectors, and Matrices}\label{sec:nonnegative}
The modulus of $\mathbf a \in \mathbb C^{1\times 1\times p}$ is defined as $\|\mathbf a\| = \sqrt{\sum^p_{i=1} |\mathbf a^{(i)}|^2 }$. We call $\mathbf a=\boldsymbol{0}$ if $\|\mathbf a\|=0$. In the sequel, the meaning of $\boldsymbol{0}$ depends on the context: it can be a tubal scalar, vector, or matrix, every entry of which is zero. Its size also depends on the context. 
\begin{definition}[Nonnegative tubal scalars]
	\label{def:nonnegative_tubal_scalars}
	We call a tubal scalar $\mathbf a\in\Kp$ nonnegative, if every entry $\mathbf a^{(i)}$ are nonnegative; we call it positive, if its modulus is not zero; we call it strongly positive, if every entry $\mathbf a^{(k)} >0$. 
	
	The set of nonnegative/positive/strongly positive tubal scalars are respectively denoted as $\Kpp$, $\Kppp$, and $\Kpppp$. 
\end{definition}
\begin{remark}\label{rmk:nonnegative_tubal_scalar}
	Note that $\Kpp = \Kppp \cup \{\boldsymbol{0} \}$.  There holds the relation
	\[
	\Kpppp\subset \Kppp\subset\Kpp \subset \Kp. 
	\]
\end{remark}

\begin{proposition}
	\label{prop:t_prod_nonnegative_tubal_scalars}
	Let $\mathbf a,\mathbf b\in \Kp$. If $\mathbf a,\mathbf b\in\Kpp$, then $\mathbf a\tprod\mathbf b\in\Kpp$; if $\mathbf a,\mathbf b\in \Kppp$, then $\mathbf a\tprod\mathbf b\in \Kppp$; if one of $\mathbf a$ or $\mathbf b$ is strongly positive, while the other one is positive, then $\mathbf a\tprod\mathbf b \in \Kpppp$.
\end{proposition}
\begin{proof}
	The results follow immediately from the definition of $\mathbf a\tprod\mathbf b$ and Definition \ref{def:nonnegative_tubal_scalars}.
\end{proof}

Proposition \ref{prop:t_prod_nonnegative_tubal_scalars} and Remark \ref{rmk:nonnegative_tubal_scalar} immediately give  that:
\begin{proposition}
	\label{prop:t_Prod_positive_tubal_scalar_times_zero_scalar}
	Let $\mathbf a\in \Knpp$ and $\mathbf b\in \Knppp$. Then $\mathbf a\tprod\mathbf b = \boldsymbol{0}$ if and only if $\mathbf a = \boldsymbol{0}$. 
\end{proposition}

\begin{definition}[Nonnegative tubal vectors]
	\label{def:nonnegative_tubal_vectors}
	We call a tubal vector $A\in\Knp$ nonnegative/positive/strongly positive, if every element of $A$ is nonnegative/positive/strongly positive.
	
	 The set of nonnegative/positive/strongly positive tubal vectors are respectively denoted as $\Knpp$, $\Knppp$, and $\Knpppp$. 
\end{definition}

From Proposition \ref{prop:t_prod_nonnegative_tubal_scalars} we have the following observation:
\begin{proposition}
	\label{prop:t_prod_positive_tubal_vector_times_nonnegative_tubal_vector}
	If $Y\in\Knppp$ and $X\in\Knpp$, $X\neq\boldsymbol{0}$, then $Y^{\top}\tprod X \in \Kppp$.
\end{proposition}
\begin{proof}
	Write $Y=(\mathbf y_i)$ and $X=(\mathbf x_i)$, $\mathbf y_i,\mathbf x_i\in \Kp$, $i=1,\ldots,n$. Then $Y^{\top}=[\mathbf y_1^{\top},\ldots,\mathbf y_n^{\top}]$, and   $Y^{\top}\tprod X = \sum^n_{i=1} \mathbf y_i^{\top}\tprod \mathbf x_i$. Since every $\mathbf y_i^{\top}\in \Kppp$ and there is at least an $\mathbf x_i\in \Kppp$, the result follows from Proposition \ref{prop:t_prod_nonnegative_tubal_scalars}. 
\end{proof}

\begin{definition}[Nonnegative tubal matrices]
	\label{def:nonnegative_tubal_matrices}
	We call a tubal matrix $A\in\Knp$ nonnegative/positive/strongly positive, if every element of $A$ is nonnegative/positive/strongly positive. 
	
	The set of nonnegative/positive/strongly positive tubal matrices are respectively denoted as $\Knnpp$, $\Knnppp$, and $\Knnpppp$. 
\end{definition}
\begin{remark}
	Under Definition \ref{def:nonnegative_tubal_matrices}, even if $\mathcal A$ is a positive tubal matrix, it can be highly sparse; for example, consider $\mathcal A=(\mathbf a_{i,j})$ where every $\mathbf a_{i,j}$ containign only one positive $\mathbf a^{(k)}_{i,j}$. 
\end{remark}

It is easy to see that:
\begin{proposition}
	\label{prop:nonnegative_tubal_matrix_times_nonnegative_tubal_vector}
	Let $\mathcal A\in\Knnpp$ and $X\in\Knpp$. Then $\mathcal A\tprod X\in \Knpp$. 
\end{proposition}

\begin{proposition}\label{prop:zero_tubal_matrix_times_positive_tubal_vector}
	Let $\mathcal A\in \Knnpp$ and $X\in \Knppp$. Then $\mathcal A\tprod X=\boldsymbol{0}$ if and only if $\mathcal A = \boldsymbol{0}$.
\end{proposition}
\begin{proof}
	This follows from Proposition \ref{prop:t_Prod_positive_tubal_scalar_times_zero_scalar}. 
\end{proof}

The following result is similar to its matrix counterpart. 
\begin{proposition}
	\label{prop:positive_tubal_matrix_times_tubal_vector}
	Let $\mathcal A\in \Knnpp$. Then $\mathcal A = \bigxiaokuohao{\mathbf a_{i,j}}\in \Knnppp$ if and only if for every $X = \bigxiaokuohao{\mathbf x_{j}}\in \Knpp$, $X\neq \boldsymbol{0}$, $\mathcal A\tprod X\in \Knppp$. 
\end{proposition}
\begin{proof}
If $\mathcal A\in \Knnppp$ and $X\in\Knpp$, $X\neq \boldsymbol{0}$, then $\mathbf a_{i,j}\in \Kppp$ for every $i,j$, and there exists at least an $\mathbf x_j \in \Kppp$. So,  $\bigxiaokuohao{\mathcal A\tprod X}_i = \mathcal A_{i,}\tprod X = \sum^n_{j=1}\mathbf a_{i,j}\tprod\mathbf x_j \in \Kppp$, where the last relation follows from Proposition \ref{prop:t_prod_nonnegative_tubal_scalars}. 

If for every $X = \bigxiaokuohao{\mathbf x_{j}}\in \Knpp$, $X\neq \boldsymbol{0}$, $\mathcal A\tprod X\in \Knppp$, while $\mathcal A\not\in \Knnppp$, assume without loss of generality that $\mathbf a_{1,1} = \boldsymbol{0}$. Take $X= \bigxiaokuohao{\mathbf x_j}$ with $\mathbf x_1\in \Kppp$ and $\mathbf x_j=\boldsymbol{0}, j=2,\ldots,n$. Then $\mathcal A_{1,}\tprod X = \sum^n_{i=1}\mathbf a_{1,j}\tprod\mathbf x_j = \boldsymbol{0}$, deducing a contradiction. 
\end{proof}

\begin{remark}
	It follows from the definition of the transpose that $\mathbf a\in \Kpp (\Kpp,~\Kppp, {\rm~or}~\Kpppp)$, then so is $\mathbf a^{\top}$. The same observations hold for the tubal vector $A^{\top}$ and the tubal matrix $\mathcal A^{\top}$.  
\end{remark}

We then define irreducible tubal matrices. Let $[n]:=\{1,\ldots,n  \}$.
\begin{definition}[Reducibility and irreducibility]
	\label{def:irreducible_tubal_matrices}
	We call a tubal matrix $\mathcal A=\bigxiaokuohao{\mathbf a_{i,j}}\in \Knnp$ reducible, if there is a nonempty proper index subset $I\subset [n]$ such that 
	\[
	\mathbf a_{i,j} = \boldsymbol{0},~\forall i\in I,~\forall j\not\in I. 
	\]
	If $\mathcal A$ is not reducible, then we call $\mathcal A$ irreducible. 
\end{definition}
\begin{remark}
	When $p=1$, the above definition boils down exactly to the reducibility/irreducibility of a matrix. 
\end{remark}
We provide some examples.
\begin{example}\label{ex:1}
	Let $\mathcal A\in \Knnp$ with $\mathcal A^{(1)} =\left[ \begin{smallmatrix}
		0 & 1 \\
		0 & 0
	\end{smallmatrix}\right]$ and $\mathcal A^{(2)} = \left[\begin{smallmatrix}
	0 & 0\\
	1 & 0
\end{smallmatrix}\right]$. It is clear that $\mathbf a_{1,2},\mathbf a_{2,1}\in\Kppp$, and so $\mathcal A$ is irreducible.  
\end{example}
\begin{example}
	Let $\mathcal A\in \Knnp$ with $\mathcal A^{(1)} =\left[ \begin{smallmatrix}
		0 & 1 \\
		0 & 0
	\end{smallmatrix}\right]$ and $\mathcal A^{(2)} = \left[\begin{smallmatrix}
		0 & 0\\
		0 & 0
	\end{smallmatrix}\right]$. Let $I=\{2\}$. Then $\mathbf a_{i,j}=0$, $\forall i\in I$, $\forall j\not\in I$, and so $\mathcal A$ is reducible. 
\end{example}

Restricting to third-order tensors, the previous reducibility/irreducibility  was defined as follows:
\begin{definition}\cite{cpz2008}
	\label{def:redu_pre}
	Let $\mathcal A\in\mathbb R^{n\times n\times n}$. $\mathcal A$ is called reducible if there is a nonempty proper index subset $I\subset [n]$ such that
	\[
	\mathcal A_{ijk} = 0,~\forall i\in I,\forall j,k\not\in I.
	\]
	If $\mathcal A$ is not reducible, then $\mathcal A$ is called irreducible.
\end{definition}

It follows from the definition of $\mathcal A^{\top}$ that:
\begin{proposition}
	\label{prop:irreducible_transpose}
	If $\mathcal A\in \Knnp$ is reducible/irreducible, then so is $\mathcal A^{\top}$.
\end{proposition}

We have the following relation between the two definitions of irreducibility. 
\begin{proposition}
	\label{prop:two_def_irreducibility}
	Let $\mathcal A\in \mathbb K^{n\times n}_n$. If it is irreducible in the sense of Definition \ref{def:redu_pre}, then it is also irreducible in the sense of our Definition \ref{def:irreducible_tubal_matrices}.
\end{proposition}
\begin{proof}
	It suffices to show that if $\mathcal A$ is reducible in the sense of our Definition \ref{def:irreducible_tubal_matrices}, then it is reducible in the sense of Definition \ref{def:redu_pre}.
	
Write $\mathcal A = \bigxiaokuohao{\mathcal A_{ijk}} = \bigxiaokuohao{\mathbf a_{i,j}}$, where $\mathcal A_{ijk}\in\mathbb R$ and $\mathbf a_{i,j}\in \Knp$.	Assume that there is $I\subset [n]$ such that $\mathbf a_{i,j}=\boldsymbol{0}$, $\forall i\in I,\forall j\not\in I$, This means that
	\[
	\mathcal A_{ijk} = \mathbf a^{(k)}_{i,j} = 0,\forall i\in I,\forall j\not\in I, \forall k\in [n],
	\] 
	wihch is clearly reducible in the sense of Definition \ref{def:redu_pre}. 
\end{proof}
\begin{remark}\label{rmk:irreducible_wider}
	Example \ref{ex:1} provides an example which is irreducible in the sense of Definition \ref{def:irreducible_tubal_matrices} while is reducible in the sense of Definition \ref{def:redu_pre}. In fact, we see that $\mathcal A_{122} = \mathbf a^{(2)}_{1,2} = 0$ in this example, showing its reducibility in the sense of Definition \ref{def:redu_pre}. 
	
	Therefore, our definition of irreducibility covers a wider range of tubal matrices than that of \cite{cpz2008}. 
\end{remark}

Similar to the matrix setting, we have the following equivalent characterizations of nonnegative irreducibility.
\begin{theorem}\label{th:reducible}
	Let $\mathcal A\in \Knnp$. Then $\mathcal A$ is reducible if and only if
	\begin{itemize}
		\item Every frontal slice $\mathcal A^{(i)}$ is a reducible matrix; more specifically, there exists a permutation matrix $P$, such that $$P\mathcal A^{(i)} P^{\top} = \left[
		\begin{matrix}
		\mathcal	B^{(i)} &\mathcal C^{(i)} \\
			\boldsymbol{0} &\mathcal D^{(i)}
		\end{matrix}
		\right],~i=1,\ldots,p,$$
		where the bottom left is a zero matrix, and $B^{(i)}$ and $D^{(i)}$ are square matrices. 
		\item There exists a permutation tensor $\mathcal P$, such that 
		\[
		\mathcal P\tprod \mathcal A \tprod \mathcal P^{\top} = \left[\begin{matrix}
\mathcal B & \mathcal C\\
\boldsymbol{0} & \mathcal D
		\end{matrix}\right],
		\]
		where the right-hand side means that $\mathcal P\tprod \mathcal A \tprod \mathcal P^{\top} $ takes a block structure, whose  bottom left is a zero tensor; $\mathcal B$ and $\mathcal D$ are square tensors. 
	\end{itemize}
\end{theorem} 
\begin{proof}
	The first claim easily follows from Definition \ref{def:irreducible_tubal_matrices}, while the second one holds by taking $\mathcal P$ with $\mathcal P^{(1)}=P$ and the other frontal slices being zero and noticing \eqref{def:permutation_tensor}. 
\end{proof}

\begin{theorem}\label{th:irreducible}
	Let $\mathcal A\in \Knnp$. Then $\mathcal A$ is nonnegative irreducible if and only if 
		  $$\bigxiaokuohao{\mathcal A + \Identity}^{(n-1)} = \overbrace{\bigxiaokuohao{\mathcal A + \Identity}\tprod \cdots \tprod \bigxiaokuohao{\mathcal A + \Identity}}^{{n-1~ {\rm times}}} \in \Knnppp.$$
\end{theorem}
\begin{proof}
Necessarity: According to Proposition \ref{prop:positive_tubal_matrix_times_tubal_vector}, it suffices to show that for any $X\in\Knpp$, $X\neq \boldsymbol{0}$, it holds that
\[
\bigxiaokuohao{\mathcal A + \Identity}^{n-1}\tprod X \in\Knppp.
\]

 For any $X\in\Knpp$, $X\neq \boldsymbol{0}$, denote $X_0:=X$. Since the associative law holds for the t-product (\cite[Lemma 3.3]{kilmer2011factorization}), It then follows that we can recursively define $$X_{k+1}:= \bigxiaokuohao{\mathcal A + \Identity}\tprod X_k = \bigxiaokuohao{\mathcal A+\Identity}^{k+1}X_0,~k=0,1,\ldots.$$  
 Proposition \ref{prop:nonnegative_tubal_matrix_times_nonnegative_tubal_vector} shows that all the $X_k$ are in $\Knpp$. Represent $X_k$ as $X_k= \bigxiaokuohao{\mathbf x_{k,i}}$, $i=1,\ldots,n$.  Denote $M_k$ as the number of zero tubal scalars of $X_k$, i.e.,  $M_k:={\rm card}\bigdakuohao{ i \mid \mathbf x_{k,i} =\boldsymbol{0} }$. Since $X_{k+1} = \mathcal A\tprod X_k + \Identity\tprod X_k = \mathcal A\tprod X_k + X_k$, by Proposition \ref{prop:nonnegative_tubal_matrix_times_nonnegative_tubal_vector}, there holds
 \[
 M_{k+1}\leq M_k \leq \cdots\leq M_0 \leq n-1,
 \]
 where the last relation follows from $X_0 \neq \boldsymbol{0}$. We then wish to show that if $M_k>0$, then $M_{k+1}<M_k$. 
 
 Otherwise, suppose that there exists a $k$ such that $M_{k+1}=M_k$; from the construction of $X_{k+1}$ and by Proposition \ref{prop:permutation_tensor_on_tubal_vector}, there exists a permutation tubal matrix $\mathcal P\in \Knnp$, such that
 \[
 \mathcal P\tprod X_{k+1} = \bigzhongkuohao{\begin{matrix}
 		A\\
 		\boldsymbol{0}
 \end{matrix}},~ \mathcal P\tprod X_k = \bigzhongkuohao{\begin{matrix}
 B\\
 \boldsymbol{0}
\end{matrix}},
 \]
 where $A,B\in \mathbb K_{p++}^{n-M_k}$, and $\boldsymbol{0}\in \mathbb K_p^{M_k}$. Since
 \[
 \mathcal P\tprod X_{k+1} = \mathcal P\tprod\mathcal A\tprod X_k + \mathcal P\tprod X_k = \mathcal P\tprod\mathcal A\tprod \mathcal P^{\top}\tprod\mathcal P\tprod X_k + \mathcal P\tprod X_k,
 \]
 where the last relation uses Proposition \ref{prop:permutation_tensor}, it follows that
 \begin{align}\label{eq:proof_irreducible_:1}
\bigzhongkuohao{\begin{matrix}
		A\\
		\boldsymbol{0}
\end{matrix}} = \left[\begin{matrix}
\mathcal B & \mathcal C\\
\mathcal D & \mathcal E
\end{matrix}\right] \tprod \bigzhongkuohao{\begin{matrix}
	B\\
	\boldsymbol{0}
\end{matrix}} + \bigzhongkuohao{\begin{matrix}
B\\
\boldsymbol{0}
\end{matrix}},
 \end{align}
 where we denote 
 \[
 \mathcal P\tprod\mathcal A\tprod\mathcal P^{\top} = \left[\begin{matrix}
 	\mathcal B & \mathcal C\\
 	\mathcal D & \mathcal E
 \end{matrix}\right], 
 \]
 where the block tensors $\mathcal B,\mathcal C,\mathcal D,\mathcal E$ are partitioned according to the size of $B$ and $\boldsymbol{0}$. In particular, $\mathcal B\in \mathbb K_p^{ (n-M_k)\times (n-M_k)  }$, and $\mathcal E\in \mathbb K_p^{M_k\times M_k}$. Comparing the left and right hand sides of \eqref{eq:proof_irreducible_:1}, we see that
 \[
 \mathcal D\tprod B = \boldsymbol{0}.
 \]
 Since $\mathcal D\in \Knnpp$ and $B\in \mathbb K_{p++}^{n-M_k}$, it follows from Proposition \ref{prop:zero_tubal_matrix_times_positive_tubal_vector} that $\mathcal D\equiv \boldsymbol{0}$. While according to Theorem \ref{th:reducible}, $\mathcal A$ is reducible, leading to a contradiction. 
 
 As a consequence, $M_{k+1}< M_k$, for any $k$ with $M_k>0$. Therefore, by the definition of $M_k$, $M_{n-1}=0$, and so $X_{n-1} = \bigxiaokuohao{\mathcal A + \Identity}^{n-1}\tprod X \in \Knppp$. 
 
 Sufficiency: Suppose on the contrary that $\mathcal A$ is reducible. By Theorem \ref{th:reducible}, there exists a permutation tubal matrix $\mathcal P\in \Knnp$, such that 
 \[
 \mathcal P \tprod\mathcal A\tprod\mathcal P^{\top} = \bigzhongkuohao{\begin{matrix}
 		\mathcal B & \mathcal C\\
 		\boldsymbol{0} & \mathcal D
 \end{matrix}},
 \]
 where $\mathcal B$ and $\mathcal D$ are square tubal matrix. From the orthogonality of $\mathcal P$, we then have
 \[
 \mathcal P\tprod\bigxiaokuohao{\mathcal A+ \Identity}\tprod\mathcal P^{\top} = \Identity + \bigzhongkuohao{\begin{matrix}
 		\mathcal B & \mathcal C\\
 		\boldsymbol{0} & \mathcal D
 \end{matrix}} = \bigzhongkuohao{\begin{matrix}
 \mathcal B  + \mathcal I_1 & \mathcal C\\
 \boldsymbol{0} & \mathcal D + \mathcal I_2
\end{matrix}},
 \]
where $\mathcal I_1$ and $\mathcal I_2$ are identity tubal matrices of proper size. Then for any $k$,
 \[
 \mathcal P\tprod \bigxiaokuohao{\mathcal A + \Identity}^{k}\tprod \mathcal P^{\top} = \bigxiaokuohao{\mathcal P\tprod\bigxiaokuohao{\mathcal A + \Identity}\tprod\mathcal P^{\top}}^k = \bigzhongkuohao{\begin{matrix}
 		\mathcal B  + \mathcal I_1 & \mathcal C\\
 		\boldsymbol{0} & \mathcal D + \mathcal I_2
 \end{matrix}}^k.
 \]
 Note that for any $k$, the right-hand side is always an upper triangular tubal matrix, which means that $\bigxiaokuohao{\mathcal A + \Identity}^k$ always contains zero tubal scalar(s), which contradicts the condition that $\bigxiaokuohao{\mathcal A + \Identity}^{n-1}\in \Knnppp$. 
\end{proof}

We present two examples to illustrate the above results.
\begin{example}
	Consider $\mathcal A =\bigxiaokuohao{\mathbf a_{i,j}}\in \mathbb K^{3\times 3 }_3$ with 
	\[
	\mathcal A^{(1)} = \bigzhongkuohao{\begin{matrix}
			0 & 1 & 0\\ 
			0 & 0 & 0\\
			0 & 0& 0
	\end{matrix}},
\mathcal A^{(2)} = \bigzhongkuohao{\begin{matrix}
		0 & 0 & 0\\
		0 & 0 & 1\\
		0 & 0& 0\\
\end{matrix}},
\mathcal A^{(3)}= \bigzhongkuohao{\begin{matrix}
		0 & 0 & 0\\
		0 & 0 & 0\\
		1 & 0& 0\\
\end{matrix}};
	\]
	that is, $\mathbf a_{1,2} = \fold{\bigzhongkuohao{\begin{smallmatrix}
				1\\0 \\ 0
	\end{smallmatrix}}}$, $\mathbf a_{2,3} = \fold{\bigzhongkuohao{ \begin{smallmatrix}
	0 \\ 1 \\ 0
\end{smallmatrix}  }}$, $\mathbf a_{3,1} = \fold{\bigzhongkuohao{ \begin{smallmatrix}
0 \\ 0\\ 1
\end{smallmatrix}  }}$, and $\mathbf a_{i,j}=\boldsymbol{0}$ for other pairs of $(i,j)$.  Let $\mathcal B = \mathcal A  + \mathcal I_{333}$.  Then $\mathbf b_{1,1}=\mathbf b_{2,2}=\mathbf b_{3,3} = \fold{\bigzhongkuohao{ \begin{smallmatrix}
1 \\0 \\ 0
\end{smallmatrix}  }}$ while $\mathbf b_{i,j}=\mathbf a_{i,j}$ for other pairs of $(i,j)$. Let $C=\mathcal B^2$. After computation we get
\[
\mathcal C^{(1)} = \bigzhongkuohao{\begin{matrix}
		1 & 2 & 0\\
		1 & 1 & 0\\
		0 & 1 & 1
\end{matrix}},
\mathcal C^{(2)} = \bigzhongkuohao{\begin{matrix}
		0 & 0 & 1 \\
		0 & 0 & 2\\
		0 & 0 & 0
\end{matrix}},
\mathcal C^{(3)} = \bigzhongkuohao{\begin{matrix}
		0 & 0 & 0\\
		0 & 0 & 0\\
		2 & 0 & 0
\end{matrix}},
\]
namely, $\mathcal C\in \Knnppp$.  Thus $\mathcal A$ is irreducible. 
\end{example}

\begin{example}
	Consider $\mathcal A =\bigxiaokuohao{\mathbf a_{i,j}}\in \mathbb K^{3\times 3 }_3$ with 
	\[
	\mathcal A^{(1)} = \bigzhongkuohao{\begin{matrix}
			0 & 1 & 0\\ 
			0 & 0 & 0\\
			0 & 0& 0
	\end{matrix}},
	\mathcal A^{(2)} = \bigzhongkuohao{\begin{matrix}
			0 & 0 & 0\\
			0 & 0 & 1\\
			0 & 0& 0\\
	\end{matrix}},
	\mathcal A^{(3)}= \bigzhongkuohao{\begin{matrix}
			0 & 0 & 0\\
			0 & 0 & 0\\
			0 & 0& 0\\
	\end{matrix}};
	\]
	that is, $\mathbf a_{1,2} = \fold{\bigzhongkuohao{\begin{smallmatrix}
				1\\0 \\ 0
	\end{smallmatrix}}}$, $\mathbf a_{2,3} = \fold{\bigzhongkuohao{ \begin{smallmatrix}
				0 \\ 1 \\ 0
	\end{smallmatrix}  }}$,  and $\mathbf a_{i,j}=\boldsymbol{0}$ for other pairs of $(i,j)$.  Let $\mathcal B = \mathcal A  + \mathcal I_{333}$.  Then $\mathbf b_{1,1}=\mathbf b_{2,2}=\mathbf b_{3,3} = \fold{\bigzhongkuohao{ \begin{smallmatrix}
				1 \\0 \\ 0
	\end{smallmatrix}  }}$ while $\mathbf b_{i,j}=\mathbf a_{i,j}$ for other pairs of $(i,j)$. Let $C=\mathcal B^2$. Then
	\[
	\mathcal C^{(1)} = \bigzhongkuohao{\begin{matrix}
			1 & 2 & 0\\
			0& 1 & 0\\
			0 & 0 & 1
	\end{matrix}},
	\mathcal C^{(2)} = \bigzhongkuohao{\begin{matrix}
			0 & 0 & 1 \\
			0 & 0 & 1\\
			0 & 0 & 0
	\end{matrix}},
	\mathcal C^{(3)} = \bigzhongkuohao{\begin{matrix}
			0 & 0 & 0\\
			0 & 0 & 0\\
			0 & 0 & 0
	\end{matrix}},
	\]
from which we see that $\mathbf c_{2,1}=\mathbf c_{3,1}=\mathbf c_{3,2}=\boldsymbol{0}$,	namely, $\mathcal C\in \Knnpp\setminus \Knnppp$.  Thus $\mathcal A$ is reducible. 
\end{example}

\begin{remark}
	Note that the irreducibility of $\mathcal A$ does not mean that $\bcirc{\mathcal A}$ is an irreducible matrix; an example is given in Example \ref{ex:pf_thm_1}. This leads to that some nice properties of nonnegative irreducible matrices do not hold for nonnegative irreducible tubal matrices. This will be studied in the next section. 
\end{remark}


 \section{Perron-Frobenius Theorem for Nonnegative Tubal Matrices}\label{sec:pf_thm}
 The t-eigenvalues and t-eigenvectors were defined as follows. 
 \begin{definition}(c.f. \cite{miao2021t,liu2020study})
 	\label{def:t_eig}
 	Let $\mathcal A\in \Knnp$. If there exists a scalar $\lambda\in \mathbb C$ and a tubal vector $X\in \mathbb C^{n\times 1\times p}$ such that
 	\[
 	\mathcal A\tprod X = \lambda X,
 	\]
 	then $\lambda$ is called a t-eigenvalue of $\mathcal A$, with its associated t-eigenvector. The set of t-eigenvalues of $\mathcal A$ is denoted as $\spec{\mathcal A}$, and the spectral radius is denoted as $\rho(\mathcal A): = \max\{ |\lambda| \mid \lambda \in \spec{\mathcal A}  \}$.
 \end{definition}
\begin{remark}\label{rmk:t_eig}
From the definition of t-product, $\mathcal A\tprod X = \lambda X$ if and only if $\bcirc{\mathcal A}\unfold{X} = \lambda \unfold{X}$.  Thus $\spec{\mathcal A} = \spec{\bcirc{\mathcal A}}$ and $\rho(\mathcal A) = \rho(\bcirc{\mathcal A})$. 
\end{remark}

Similar to the matrix case, in this paper we call $X$ in Definition \ref{def:t_eig} the right t-eigenvector of $\mathcal A$, and we can define left t-eigenvectors. 
\begin{definition}
	If there is a scalar $\lambda\in C$ and a tubal vector $Y\in \mathbb C^{n\times 1\times p}$ such that
	\[
	\mathcal A^{\top} \tprod Y = \lambda Y,
	\]
	then $Y$ is called a right t-eigenvector of $\mathcal A$. 
\end{definition}
\begin{remark}
	Using \eqref{prop:transpose}, $\unfold{Y}$ is a left eigenvector of $\bcirc{\mathcal A}^{\top}$. Thus $\spec{\mathcal A^{\top}} = \spec{\mathcal A}$, and $\rho(\mathcal A^{\top}) = \rho(\mathcal A)$.
\end{remark}

When context permits, we use the t-eigenvector to mean a right t-eigenvector.

\begin{theorem}[Perron-Frobenius theorem for nonnegative tubal matrices]\label{th:weak_pf_thm_tubal_matrix}
	Let $\mathcal A\in \Knpp$, $\mathcal A\neq \boldsymbol{0}$ be a nonnegative tubal matrix. Then $\rho(\mathcal A)$ is a t-eigenvalue of $\mathcal A$, with a nonnegative t-eigenvector $X\in\Knpp$, $X\neq \boldsymbol{0}$ corresponding to it. 
\end{theorem}
\begin{proof}
	Since $\mathcal A\in\Knnpp$, $\bcirc{\mathcal A} \in\mathbb R^{np\times np}$ is also a nonnegative matrix, By the Perron-Frobenius theorem for nonnegative matrices, $\rho(\bcirc{\mathcal A})$ is an eigenvalue of $\bcirc{\mathcal A}$, with a nonnegative eigenvector $\mathbf x\in\mathbb R^{np}$, $\mathbf x\neq \boldsymbol{0}$ corresponding to it. Noticing Remark \ref{rmk:t_eig} and folding $\mathbf x$ to a tubal vector $X\in\Knpp$ give the desired result. 
\end{proof}
\begin{remark}
	Theorem \ref{th:weak_pf_thm_tubal_matrix} also holds for the right t-eigenvectors of $\mathcal A$. 
\end{remark}

Next we consider PF theorem for nonnegative irreducible tubal matrices. Some preparations are needed.

\begin{lemma}\label{lem:irreducible_inequality_2_equality}
	Let $\mathcal A\in \Knnpp$, $\mathcal A\neq \boldsymbol{0}$ admit a positive left t-eigenvector $Y\in \Knppp$ corresponding to $\rho(\mathcal A)$. If $X\in\Knp$, $X\neq \boldsymbol{0}$ satisfies 
	$\mathcal A\tprod X-\rho(\mathcal A) X \in\Knpp$, then $\mathcal A\tprod X=\rho(\mathcal A)X$, 
	i.e, $X$ is a t-eigenvector corresponding to $\rho(\mathcal A)$.
\end{lemma}
\begin{proof}Since $\mathcal A\tprod X-\rho(\mathcal A) X \in\Knpp$, if it is identically zero, then we are done. Otherwise, since $Y\in\Knppp$,  by Proposition \ref{prop:t_prod_positive_tubal_vector_times_nonnegative_tubal_vector}
we have $Y^{\top}\tprod \bigxiaokuohao{\mathcal A \tprod X- \rho(\mathcal A)X} \in \Kppp$.  However,  as $Y^{\top}\tprod \mathcal A = \rho(\mathcal A)Y^{\top}$, we have $Y^{\top}\tprod \bigxiaokuohao{\mathcal A \tprod X- \rho(\mathcal A)X} =\boldsymbol{0}$, deducing a contradiction. Thus the assertion holds. 
\end{proof}

For $\mathcal A\in \mathbb C^{n\times n\times p}$, denote $|\mathcal A|\in\Knnpp$ as the tubal matrix whose every element is the magnitude of the corresponding element of $\mathcal A$.  The same notation applies also to tubal vectors and scalars, and the usual matrices, vectors, and scalars. 
The following lemma is useful.
\begin{lemma}
	\label{lem:magnitude_larger}
	If $\mathcal A\in \mathbb C^{n\times n\times p}$ and $X\in \mathbb C^{n\times 1\times p}$, then $ \bigjueduizhi{\mathcal A}\tprod\bigjueduizhi{X}- \bigjueduizhi{\mathcal A\tprod X} \in \Knpp$.
\end{lemma}
\begin{proof}
By noting the relation between $\mathcal A\tprod X$ and $\bcirc{\mathcal A} \unfold{X}$, and realizing that $\bigjueduizhi{ \bcirc{\mathcal A}\unfold{X}  } \leq \bigjueduizhi{ \bcirc{\mathcal A} }\bigjueduizhi{ \unfold{X} }$, the result follows. 
\end{proof}

\begin{theorem}[Perron-Frobenius theorem for nonnegative irreducible tubal matrices]
	\label{th:pf_thm_irreducible_tubal_matrix}
	Let $\mathcal A\in \Knnpp$ be  a nonnegative irreducible tubal matrix. Then 
	\begin{enumerate}
		\item $\rho(\mathcal A)$ is a t-eigenavlue of $\mathcal A$;
		\item $\rho(\mathcal A)>0$;
		\item Every  nonnegative t-eigenvector $X\in\Knpp$, $X\neq \boldsymbol{0}$,  must be a positive t-eigenvector, i.e., $X\in\Knppp$;
		\item There exists a positive t-eigenvector $X\in\Knppp$ corresponding to $\rho(\mathcal A)$;
		\item All the t-eigenvectors corresponding to $\lambda$ with $|\lambda|=\rho(\mathcal A)$ do not contain zero tubal scalars;
		\item All the real t-eigenvectors corresponding to $\rho(\mathcal A)$ cannot take the following form: Some of its tubal elements are positive, while the other ones are not;
		\item if $\lambda$ is a t-eigenvalue with a strongly positive t-eigenvector $Y\in \Knpppp$, then $\lambda = \rho(\mathcal A)$.
	\end{enumerate}
All the above results hold for the left t-eigenvectors. 
\end{theorem}
\begin{proof}
	Item 1 follows from Theorem \ref{th:weak_pf_thm_tubal_matrix}. 
	
	Let $\mathcal A\tprod X=\rho(\mathcal A) X$ with $X\in\Knpp$, $X\neq \boldsymbol{0}$ the corresponding nonnegative t-eigenvector. To show $\rho(\mathcal A)>0$, it suffices to show that $\mathcal A\tprod X\neq \boldsymbol{0}$.  Represent $\mathcal A = (\mathbf a_{i,j}), X=(\mathbf x_i)$, $\mathbf a_{i,j},\mathbf x_i\in \Knpp$. Let $I:=  \bigdakuohao{ i \mid \mathbf x_i=\boldsymbol{0}  }$, and $\bar I = [n]\setminus I$. If $\mathcal A\tprod X \equiv \boldsymbol{0}$, then
	\[
	\sum^n_{j=1}\mathbf a_{i,j}\tprod \mathbf x_j = \boldsymbol{0},~\forall i \in [n],
	\]
	which together with the nonnegativity of $\mathbf a_{i,j}$ and $\mathbf x_j$ gives that
	\[
	\mathbf a_{i,j}\tprod\mathbf x_j = 0,\forall i\in [n], \forall j\in [n]. 
	\]
	In particular, when $j\in \bar I$, $\mathbf x_j \in \Knppp$, which by Proposition \ref{prop:t_Prod_positive_tubal_scalar_times_zero_scalar} implies that
	\[
	\mathbf a_{i,j} = \boldsymbol{0}, ~\forall i \in [n], \forall j \in \bar I,
	\]
	which contradicts the irreducibility of $\mathcal A$. Thus $\mathcal A\tprod X\neq \boldsymbol{0}$, and so $\rho(\mathcal A)>0$.  This proves Item 2.
	
	Let $X\in\Knpp$, $X\neq \boldsymbol{0}$ be a nonnegative t-eigenvector corresponding to a t-eigenvalue $\lambda$. Similar to the proof of Item 2 we have $\lambda>0$. 
	If $X\not\in \Knppp$, then let $I$ and $\bar I$ be defined as those in Item 2. Then 
	\[
	\boldsymbol{0} = \lambda\mathbf x_i = \sum^n_{j=1}\mathbf a_{i,j}\tprod\mathbf x_j,~\forall i\in I~\Leftrightarrow \boldsymbol{0}=\mathbf a_{i,j}\tprod\mathbf x_j,\forall i\in I.
	\]
	Since when $j \in \bar I$, $\mathbf x_j\in \Knppp$, it follows again from Proposition \ref{prop:t_Prod_positive_tubal_scalar_times_zero_scalar} that
	\[
	\mathbf a_{i,j} = \boldsymbol{0},~\forall i\in I,\forall j\in \bar I,
	\]
	deducing a contradiction. Thus $I$ is empty and so $X\in \Knppp$. This proves Item 3.
	
	Item 4 follows from Theorem \ref{th:weak_pf_thm_tubal_matrix} and Item 3. 
	
	To prove Item 5, let $X \in \mathbb C^{n\times 1\times p}$ be a t-eigenvector with t-eigenvalue $\lambda\in \mathbb C$ and $|\lambda|=\rho(\mathcal A)$, such that $\mathcal A\tprod X = \lambda X$. By Lemma \ref{lem:magnitude_larger},
	\begin{align}\label{eq:proof_pf_thm:1}
	\mathcal A\tprod\bigjueduizhi{X} -   \rho(\mathcal A)\bigjueduizhi{X} = \mathcal A\tprod \bigjueduizhi{X} - \bigjueduizhi{\lambda X} = \mathcal A\tprod \bigjueduizhi{X} - \bigjueduizhi{\mathcal A\tprod X} \in \Knpp. 
	\end{align}
	It follows from Proposition \ref{prop:irreducible_transpose} that $\mathcal A^{\top}$ is also nonnegative irreducible; then Item 4 tells us that there is a positive left t-eigenvector $Y\in\Knppp$  corresponding to $\rho(\mathcal A)$ of $\mathcal A^{\top}$. Applying Lemma \ref{lem:irreducible_inequality_2_equality} to \eqref{eq:proof_pf_thm:1} shows that
	\[
	\mathcal A\tprod \bigjueduizhi{X} = \rho(\mathcal A) \bigjueduizhi{X},
	\]
	i.e., $|X|$ is a nonnegative t-eigenvector of $\mathcal A$. Then Item 3 shows that $|X|\in\Knppp$, namely, $|X|$ does not contain zero tubal scalars. The same situation also holds for $X$. 
	
	We then prove Item 6. Let $X\in \Knp$ with $\mathcal A\tprod X = \rho(\mathcal A)X$. 
	Write $X=(\mathbf x_i)$. Item 5 tells us that $\mathbf x_i\neq \boldsymbol{0}$, $i=1,\ldots,n$. If some of $\mathbf x_i \in \Kppp$ while the other ones are not, 
	without loss of generality, we assume that
		\begin{align*}
	\mathbf x_i	\begin{cases}
		\in \Kppp,& i \in I \subset [n] , I\neq \emptyset,\\
			\not\in \Kpp,& i \in \bar I.
		\end{cases}
	\end{align*} 
On the other hand, the proof of Item 5 shows that 
\[
\mathcal A\tprod \bigjueduizhi{X} = \rho(\mathcal A)\bigjueduizhi{X} = \bigjueduizhi{\rho(\mathcal A)X} = \bigjueduizhi{\mathcal A\tprod X}. 
\]
Then, if $i\in \Kppp$, we have
\[
\bigxiaokuohao{\mathcal A\tprod X}_i = \rho(\mathcal A)\mathbf x_i = \rho(\mathcal A) \bigjueduizhi{\mathbf x_i} = \bigjueduizhi{ \bigxiaokuohao{\mathcal A\tprod X}_i  } = \bigxiaokuohao{\mathcal A\tprod \bigjueduizhi{X}}_i,
\]
which means that
\[
\sum^n_{j=1}\mathbf a_{i,j}\tprod \bigxiaokuohao{ \bigjueduizhi{\mathbf x_j} -\mathbf x_j} = \boldsymbol{0}, ~\forall i\in I. 
\]
It follows from $|\mathbf x_j| - \mathbf x_j\in \Kpp$ that
\[
\mathbf a_{i,j}\tprod\bigxiaokuohao{ |\mathbf x_j| - \mathbf x_j  } = \boldsymbol{0},~\forall i\in I. 
\]
Note that $\mathbf x_j\not\in \Kpp$ when $j\in \bar I$; i.e.,  there exist some $\mathbf x_j^{(k)}<0$, which implies that $|\mathbf x_j| - \mathbf x_j \in \Kppp$.   Proposition \ref{prop:t_Prod_positive_tubal_scalar_times_zero_scalar} then tells us that
\[
\mathbf a_{i,j} = \boldsymbol{0},~\forall i\in I,\forall j\in \bar I,
\]
contradicting the irreducibility of $\mathcal A$.  This completes the proof of Item 6.
	
To show Item 7,	let $\mathcal A\tprod Y = \lambda  Y$ with $Y\in\Knpppp$. Similar to the   proof of Item 2, we have $\lambda>0$. Define $$\delta_{X,Y}:= \sup_{s\geq 0} \left\{s   \mid Y- s X \in \Knpp  \right\}.$$      Then $\delta_{X,Y}>0$ since $Y\in \Knpppp$. Thus
	\[
	\lambda Y - \rho(\mathcal A) \delta_{X,Y} X = \mathcal A\tprod\bigxiaokuohao{ Y - \delta_{X,Y} X} \in  \Knpp,  
	\]
	where the last relation follows from Proposition \ref{prop:nonnegative_tubal_matrix_times_nonnegative_tubal_vector}.  Since $\lambda>0$, this means that $Y-\frac{\rho(\mathcal A)}{\lambda} \delta_{X,Y} X \in \Knpp$, which together with the definition of $\delta_{X,Y}$ implies that $\rho(\mathcal A) / \lambda \leq 1$. Thus we must have $\rho(\mathcal A) = \lambda$. 
\end{proof}
 
We present an example to illustrate Theorem \ref{th:pf_thm_irreducible_tubal_matrix}.
\begin{example}
	\label{ex:pf_thm_1}
	Consider $\mathcal A\in \mathbb K_2^{2\times 2}$, with 
	\[
	\mathcal A^{(1)} = \bigzhongkuohao{\begin{matrix}
			0 &  0 \\
			0 &0
	\end{matrix}},~ \mathcal A^{(2)} = \bigzhongkuohao{\begin{matrix}
	0 &  1 \\
	1 &0
\end{matrix}}.
	\]
	Then $\bcirc{\mathcal A}$ is
	\[
	\bcirc{\mathcal A} = \bigzhongkuohao{\begin{matrix}
			0 & 0 & 0 & 1\\
			0& 0& 1& 0\\
			0 & 1 & 0 & 0\\
			1& 0 & 0 & 0
	\end{matrix}}.
	\]
	Note that  $\mathcal A$ is an irreducible tubal matrix, but $\bcirc{\mathcal A}$ is a reducible matrix. Since $\mathcal A^{\top}$ is symmetric, its left and right t-eigenvectors are the same.  
	
	$\mathcal A$ then has four t-eigenvalues: $1$ (with multiplicity $2$) and $-1$ (with multiplicity $2$). Thus $\rho(\mathcal A)=1>0$ is an t-eigenvalue. The t-eigenvectors are
	\begin{align*}
&	A ~{\rm with} ~\mathbf a_1 = [0, 1], \mathbf a_2 = [-1,0];~ &B {\rm~with}~ \mathbf b_1 = [1,0], \mathbf b_2 = [0 -1];\\
& C~{\rm with}~\mathbf c_1 = [1,0],\mathbf c_2 = [0,1]; ~&D~{\rm with}~\mathbf d_1 =[0,1]; \mathbf d_2 = [1,0].
	\end{align*}
Here $A$ and $B$ correspond to $-1$, while $C$ and $D$ correspond to $\rho(\mathcal A)=1$. Every linear combination of $A$ and $B$ is also a t-eigenvector of $-1$, and every linear combination of $C$ and $D$ is also a t-eigenvector of $1$. 

From the structure of $C$ and $D$, we can see that  $C$ and $D$ are positive; moreover,  all the nonnegative t-eigenvectors must be positive. These illustrate Items 3 and 4 of Theorem \ref{th:pf_thm_irreducible_tubal_matrix}.
These also tell us  that the positive t-eigenvectors corresponding to $\rho(\mathcal A)$ might not be unique. 

Besides $C$ and $D$, we can also see that $A$ and $B$ do not contain zero tubal scalars, and any linear combination of $A$ and $B$ does not contain zero tubal scalars either. This confirms Item 5 of Theorem \ref{th:pf_thm_irreducible_tubal_matrix}. 

From again the structure of $C$ and $D$, we can check that any linear combination of $C$ and $D$ is either positive, or does not belong to $\mathbb K^2_{2+}$, as proved in Item 6 of Theorem \ref{th:pf_thm_irreducible_tubal_matrix}.

To show Item 7 of Theorem \ref{th:pf_thm_irreducible_tubal_matrix}, let $C=\alpha A + \beta B$ where $\alpha,\beta \in \mathbb R$. It is clearly seen that if $\mathbf c_1$ is strongly positive, then $\mathbf c_2$ cannot be; if $\mathbf c_2$ is strongly positive, then $\mathbf c_1$ cannot be. While any positive linear combination of $C$ and $D$ is a strongly positive t-eigenvector. Thus any strongly positive t-eigenvector can only correspond to $\rho(\mathcal A)$. 
\end{example}

We then discuss the differences between Theorem \ref{th:pf_thm_irreducible_tubal_matrix} and PF theorem for nonnegative irreducible matrices. Recall the PF theorem:
\begin{theorem}
	\label{th:pf_thm_nonnegative_irreducible_matrix}
	Let $A\in\mathbb R^{n\times n}$ be nonnegative irreducible. Then
	\begin{enumerate}
		\item $\rho(A)>0$ is an eigenvalue;
		\item There is a positive eigenvector corresponding to $\rho(A)$;
		\item If $\lambda$ is an eigenvalue with a nonnegative eigenvector, then $\lambda=\rho(A)$;
		\item $\rho(A)$ is a simple root.
	\end{enumerate}
\end{theorem}
Comparing Theorem \ref{th:pf_thm_irreducible_tubal_matrix} with Theorem \ref{th:pf_thm_nonnegative_irreducible_matrix}, we see that Items 1 and 2 of Theorem \ref{th:pf_thm_nonnegative_irreducible_matrix} are inherited by nonnegative irreducible tubal matrices. Item 3 of Theorem \ref{th:pf_thm_nonnegative_irreducible_matrix} cannot be completely generalized to the tubal matrix case, where Item 7 of Theorem \ref{th:pf_thm_irreducible_tubal_matrix} gives a weaker version. Item 4 of Theorem \ref{th:pf_thm_nonnegative_irreducible_matrix} does not hold for tubal matrices, as illustrated in Example \ref{ex:pf_thm_1}; what is worse is that $\rho(\mathcal A)$ may have t-eigenvectors that are neither positive nor negative. The only thing currently we can make sure is that the real t-eigenvectors corresponding to $\rho(\mathcal A)$ is either positive, or cannot contain nonnegative tubal scalars, as shown in Item 6 of Theorem \ref{th:pf_thm_irreducible_tubal_matrix}. 


Theorem \ref{th:pf_thm_irreducible_tubal_matrix} can be enhanced if there is an additionally relatively mild assumption is satisfied. 
\begin{theorem}[Enhanced PF theorem]
	\label{th:pf_thm_irreducible_tubal_matrix_enhanced}
	Let  $\mathcal A\in \Knnpp$ be a nonnegative irreducible tubal matrix. If there exists a strongly positive tubal scalar $\mathbf a_{i,j}$ in $\mathcal A$, Then
	\begin{enumerate}
		\item $\rho(\mathcal A)>0$ is a t-eigenvalue;
				\item Every nonzero nonnegative t-eigenvector must be strongly positive; 
		\item There is a strongly positive t-eigenvector corresponding to $\rho(\mathcal A)$; such a strongly positive t-eigenvector is unique up to a multiplicative constant; 
		\item If $\lambda$ is a t-eigenvalue with a nonzero nonneagtive t-eigenvector, then $\lambda = \rho(\mathcal A)$.
	\end{enumerate}
\end{theorem}
\begin{proof}
	Item 1 has been proved in Theorem \ref{th:pf_thm_irreducible_tubal_matrix}.
	
	Let $X=\bigxiaokuohao{\mathbf x_i}\in\Knpp$, $X\neq 0$ with $\mathcal A\tprod X = \lambda X$. The proof of Item 2 of Theorem \ref{th:pf_thm_irreducible_tubal_matrix} shows that $\lambda>0$.  Item 3 of Theorem \ref{th:pf_thm_irreducible_tubal_matrix} shows that $X\in\Knppp$. We then show that $X\in \Knpppp$ in the current setting. Since $X\in \Knppp$ and there exists an $\mathbf a_{i,j}\in \Kpppp$, Proposition \ref{prop:t_prod_nonnegative_tubal_scalars} shows that there exists an $\mathbf x_k\in\Kpppp$. Therefore, let $I:= \bigdakuohao{ i \mid  \mathbf x_i \in \Kppp\setminus \Kpppp   }$ and $\bar I = [n]\setminus I$. Then $\bar I \neq \emptyset$.    We will prove that $I=\emptyset$. If not, consider   $i\in I$,
	\[
	\sum^n_{j=1}\mathbf a_{i,j}\tprod\mathbf x_j = \lambda\mathbf x_i, \forall i\in I.
	\]
	Since when $j\in \bar I$, $\mathbf x_j\in \Kpppp$;  if $\mathbf a_{i,j}\neq \boldsymbol{0}$, then Proposition \ref{prop:t_prod_nonnegative_tubal_scalars} shows that $\mathbf a_{i,j}\tprod \mathbf x_j\in\Kpppp$, and so  the LHS above is also in $\Kpppp$,  namely, $\mathbf x_i\in \Kpppp$, which contradicts the definition of $I$. Thus the only case can happen is that 
	\[
	\mathbf a_{i,j}=\boldsymbol{0},~\forall i\in I,\forall j\in \bar I. 
	\]
	However, this also leads to a contradiction. Thus, $I=\emptyset$, and so $X\in\Knpppp$.
	
To show	Item 3, note that the existence of a strongly positive t-eigenvector follows from Item 4 of Theorem \ref{th:pf_thm_irreducible_tubal_matrix} and Item 2 above.  Assume that $\mathcal A\tprod X = \rho(\mathcal A) X$ and $\mathcal A\tprod Y = \rho(\mathcal A) Y$ with $X,Y\in \Knpppp$, and $X\neq \alpha Y$, $\alpha >0$.  Denote $\delta_{X,Y}$ the same as that in the proof of Item 7 of Theorem \ref{th:pf_thm_irreducible_tubal_matrix}. Then $\delta_{X,Y}>0$. Denote $Z:= Y- \delta_{X,Y} X$. It is clear that $Z\in \Knpp\setminus\Knpppp$, namely, there is at least a tubal scalars $\mathbf z_i\not\in \Kpppp$. However, $Z$ is a nonnegative t-eigenvector corresponding to $\rho(\mathcal A)$. Thus Item 2 above tells us that the only case now can happen is $Z=\boldsymbol{0}$, i.e., $X=\alpha Y$ for some $\alpha>0$. This proves the uniqueness of the strongly positive t-eigenvector. 
	
	Item 4 follows from Item 7 of Theorem \ref{th:pf_thm_irreducible_tubal_matrix} and Item 2  above. 
\end{proof}
Comparing with the PF theorem for nonnegative irreducible matrices,  all the conclusions except the simplicity of $\rho(\mathcal A)$ hold now.

\section{Conclusions}\label{sec:con}
This paper treats third-order tensors as tubal matrices, and defines nonnegative/positive/strongly positive tubal scalars/vectors/matrices. The meaning of positivity here is different from its usual sense. For example, a positive tubal matrix might be very sparse. 
 Based on the nice properties of the t-product, we derive several properties for nonnegative tubal scalars/vectors/matrices. The irreducibility is defined with equivalent characterizations given. Some conclusions of the Perron-Frobenius theorem for nonnegative irreducible matrices have been generalized to nonnegative irreducible tubal matrices setting, while some fail. With an additionally relatively mild assumption, the obtained conclusions are enhanced, which are more close to the original PF theorem. 
 
{\footnotesize \section*{Acknowledgment} This work was supported by the National Natural Science Foundation of China Grant 11801100,   the Fok Ying Tong Education Foundation Grant 171094, and the special foundation for Guangxi Ba Gui Scholars.
}

   \bibliography{tensor,TensorCompletion,orth_tensor,random_tensor,t_prod,bib_tensor}
  \bibliographystyle{plain}

   \appendix

\end{document}